\documentclass[a4paper,12pt]{article}
\usepackage{geometry} 
\usepackage{amssymb,amsmath}
\date{April 2011}

\newtheorem{theorem}{Theorem}
\newtheorem{lemma}[theorem]{Lemma}
\newtheorem{proposition}[theorem]{Proposition}
\newtheorem{corollary}[theorem]{Corollary}
\newtheorem{definition}[theorem]{Definition}
\newtheorem{example}[theorem]{Example}
\newtheorem{remark}[theorem]{Remark}
\newenvironment{proof}{}{}
\renewcommand{\proof}[1][]{\def\ptemp{#1}\noindent 
\textbf{Proof.}\ifx\ptemp\empty\else\,[#1]\fi \hspace{.4em}}

\makeatletter\def\operatorname#1{\mathop{\operator@font #1}\nolimits}\makeatother

\newcommand{\B}{\mathbb{B}}
\newcommand{\C}{\mathbb{C}}
\newcommand{\R}{\mathbb{R}}
\newcommand{\Z}{\mathbb{Z}}
\newcommand{\map}[1][]{\stackrel{#1}{\longrightarrow}}
\newcommand{\suchthat}{\mathop{\,\vert\,}}
\renewcommand{\H}{\mathcal{H}}
\newcommand{\half}{{\textstyle{\frac12}}}
\newcommand{\Id}{\mathrm{Id}}

\newcommand{\Det}{\operatorname{det}}
\newcommand{\Tr}{\operatorname{Trace}}
\newcommand{\fhi}{\varphi}

\newcommand{\h}{\mathfrak{h}}

\newcommand{\cc}{{\rm\scriptstyle c}}
\newcommand{\Mpc}{Mp^{\cc}}
\newcommand{\mpc}{\mathfrak{mp}^{\cc}}

\newcommand{\MUc}{MU^{\cc}}

\renewcommand{\sp}{\mathfrak{sp}}

\renewcommand{\u}{\mathfrak{u}}

\newcommand{\setdef}[2]{\left\{{#1}\suchthat{#2}\right\}}

\newcommand{\onehalf}{\mbox{$\frac{\scriptstyle 1}{\scriptstyle 2}\,$}}

\let\ol\overline

\let\wt\widetilde

\def\ftnote#1{\def\footnotemark{}\footnote{#1}\setcounter{footnote}{0}}

\title{Symplectic Dirac Operators and $\Mpc$-structures}

\author{Michel Cahen$^1$, Simone Gutt$^{1,2}$ and John Rawnsley$^3$
\ftnote{This work benefited from an Action de Recherche Concert\'ee de 
la Communat\'e fran\c caise de Belgique.}
\ftnote{$^1$D{\'e}partement de Math{\'e}matique, Universit{\'e} Libre de
Bruxelles, Campus Plaine, C. P. 218, Boulevard du Triomphe, B-1050
Bruxelles, Belgium.}
\ftnote{$^2$Universit\'e de Metz, D{\'e}partement de Math{\'e}matique, Ile
du Saulcy, F-57045~Metz Cedex 01, France.}
  \ftnote{$^3$Mathematics Institute, University of Warwick, Coventry \ CV4 \ 7AL, UK.}
\ftnote{mcahen@ulb.ac.be, sgutt@ulb.ac.be, j.rawnsley@warwick.ac.uk}}

\begin{document}
\maketitle

\thispagestyle{empty}\enlargethispage{1cm}

\let\savedabstractname=\abstractname
\renewcommand{\abstractname}{Dedication}
\begin{abstract}
\noindent Michel Cahen met Josh Goldberg fifty years ago in King's College, London.
They shared a common interest in the holonomy group of pseudo Riemannian
manifolds. Michel enormously appreciated the kindness of an established
scientist for a young PostDoc and admired his constant effort for deep
understanding of scientific questions. It is with great pleasure that the
three authors dedicate this work (partially inspired by physics) to Josh
Goldberg.
\bigskip\bigskip\bigskip

      \begin{center}\bfseries \savedabstractname\end{center}

\noindent Given a symplectic manifold $(M,\omega)$ admitting a metaplectic structure,
and choosing  a positive  $\omega$-compatible almost complex structure $J$
and a linear connection $\nabla$ preserving $\omega$ and $J$, Katharina  and Lutz
Habermann have constructed two Dirac operators $D$ and  ${\wt{D}}$
acting on sections of  a bundle of symplectic spinors. They have shown that
the commutator $[ D, {\wt{D}}]$ is an elliptic operator preserving
an  infinite number of finite dimensional subbundles. We extend the
construction of symplectic Dirac operators to any symplectic manifold,
through the use of $\Mpc$ structures. These exist on any symplectic
manifold and equivalence classes are parametrized by elements in
$H^2(M,\Z)$. For any $\Mpc$ structure,   choosing  $J$ and  a linear
connection $\nabla$ as before, there are two natural Dirac operators,
acting on the sections of a spinor bundle, whose commutator $\mathcal{P}$
is elliptic. Using the Fock description of the spinor space allows the
definition of a notion of degree and the construction of a dense family of
finite dimensional  subbundles; the operator $\mathcal{P}$ stabilizes the
sections of each of those.

 \end{abstract}

\bigskip
\bigskip

\newpage

\section{Introduction}
\label{intro}

Symplectic spinors were introduced by Kostant \cite{refs:Kostant} as a
means of constructing half-forms for geometric quantization, a notion
similar to half-densities but better suited to the symplectic category.
In his paper a metaplectic structure is needed to define his spinors, a
notion which is topologically the same as admitting a spin structure.
This rules out important examples of symplectic manifolds which are not
spin such as $\C P^2$. Instead, in \cite{refs:ForgerHess,refs:RobRaw} it
is shown that symplectic manifolds always admit $\Mpc$-structures 
(the symplectic analogue of $Spin^\cc$) and thus there always exist symplectic
 spinors (in the $\Mpc$ sense) .

By further choosing suitable connections, there exist  Dirac
operators on any symplectic manifold. 
The connections required are a pair consisting of a
connection in the tangent bundle preserving the symplectic 2-form (it
can have torsion) as in the metaplectic case and additionally a
Hermitean connection in an associated complex Hermitean line bundle. In
geometric quantization this line bundle is related to the prequantum
line bundle and there is a topological constraint on its Chern class. In
the case of constructing symplectic spinors on a general symplectic
manifold this line bundle can be chosen arbitrarily, and can even be
taken to be trivial whilst in the metaplectic case it has to be a square
root of the symplectic canonical bundle (which may not have even degree).

Apart from our use of $\Mpc$ structures, there is a second place where our
methods differ from those of \cite{refs:Habermanns}. The construction of
symplectic spinors in \cite{refs:Habermanns} is based on the Schr\"odinger
representation of the Heisenberg group whilst we use the Fock picture
realised as a holomorphic function Hilbert space. This contains polynomial
subspaces and so has a dense subspace graded by degree and this degree
transfers to the fibres when a positive compatible almost complex structure
is chosen. This makes the bundle of symplectic spinors into a direct sum of
finite dimensional subbundles $S^k T'M^* \otimes L$ in a very explicit
way. Choosing a connection in the tangent bundle which preserves both the
symplectic form and the positive compatible almost complex structure,  the
Dirac operator $D$ splits $D = D' +  D''$ where $D'$ involves $(1,0)$
derivatives and $D''$ involves $(0,1)$ derivatives. An examination of the
combination of Clifford multiplications involved in $D'$ shows that it has
coefficients which are creation operators in Fock space and $D''$ has
coefficients which are annihilation operators.

The second Dirac operator $\wt{D}$ defined in \cite{refs:Habermanns},
which is formed using the metric instead of the symplectic form to give
the isomorphism of tangents with cotangents, can then be written $\wt{D}
= -iD' + i D''$. Hence the second order operator
$\mathcal{P}=i[\wt{D},D] = 2[D',D'']$ will clearly preserve degrees and
acts as $-D''D'$ in degree zero, i.e.{} in the line bundle $L$. We show,
extending the results of \cite{refs:Habermanns} to our context, that the
symbol of this second order operator coincides with the symbol of the
Laplacian $\nabla^*\nabla$ and we give a Weitzenb\"ock-type formula.

In the framework of Spin geometry on an oriented  Riemannian manifold
$(M,g)$ each tangent space can be modelled on an  oriented Euclidean
vector space $(V,\tilde{g})$, and one needs the following ingredients:
\begin{itemize}
\item[-] a spinor space $S$, which is an irreducible representation $cl$ of
the Clifford Algebra $Cl(V,G)$ (the associative unital algebra generated by
 $V$ with $u\cdot v+v\cdot u =-2\tilde{g}(u,v) 1$);
\item[-] a  group $G$ (the group $Spin$ or $Spin^\cc$) which is a central
extension of the group $SO(V,\tilde{g})$ of linear isometries of 
$(V,\tilde{g})$, with a surjective homomorphism $\sigma: G \rightarrow
SO(V,\tilde{g})$,  and with a representation $\rho$ on the spinor space $S$
such that 
\[
\rho(h)\circ cl(v)\cdot \circ
\rho(h^{-1})=cl(\sigma(h)v) \text{ for all } h\in G \text{ and } v\in V;
\]
\item[-] a G-principal bundle $P$ on $M$ with a map $\phi:P\rightarrow
B(M,g)$ on the $SO(V,\tilde{g})$-principal bundle of oriented orthonormal
frames of $(M,g)$ such that
\[
\phi(\xi\cdot h)= \phi(\xi)\cdot \sigma(h) 
\text{ for all } \xi\in P,\ \ h\in G;
\]
\item[-] spinor fields on $M$  are sections of the associated  bundle
$\mathcal{S}:= P\times_{G,\rho} S$; Clifford multiplication yields a map
$Cl$ from the tangent bundle $TM=P\times_{G,\sigma} V$ to the bundle of
endomorphisms of $\mathcal{S}$;
\item[-] a connection on $P$ (which is a $1$-form $\alpha$ on $P$ with
values in the Lie algebra $\mathfrak{g}$ of $G$); this induces a covariant
derivative $\nabla$ of spinor fields; one assumes that the projection of
$\alpha$ on the Lie algebra of $SO(V,\tilde{g})$ is the pullback by 
$\phi:P\rightarrow B(M,g)$ of the Levi Civita connection on $B(M,g)$;
\item[-] one defines the Dirac operator acting on spinor fields as the
contraction (using the metric $g$) of the Clifford multiplication and the
covariant derivative:
\[
D\psi =\sum_a Cl(e_a)\nabla_{e_a}\psi
\]
where $e_a(x)$ is an orthonormal frame at $x\in M$.
\end{itemize}

In the framework of symplectic geometry and $\Mpc$ structures we present
all the corresponding steps. For a symplectic manifold $(M,\omega)$ each
tangent space has a structure of symplectic vector space $(V,\Omega)$
and the symplectic Clifford algebra is the unital associative algebra generated
by $V$ with the relations $u\cdot v -v\cdot u=
\frac{i}{\hbar}\Omega(u,v) 1$. An irreducible representation of this
algebra corresponds to an irreducible representation of the Lie algebra
of the Heisenberg group with prescribed central character equal to
$-\frac{i}{\hbar}$. In Section \ref{section:Fock}, we describe the Fock
representation space of  the Heisenberg group and we present possible
typical fibres of the symplectic spinor bundle, e.a. smooth vectors of
this representation.

In Section
\ref{sec:symp+mpc}, we describe  the group $\Mpc$ which is a circle
extension of the symplectic group. It has a character $\eta:\Mpc\rightarrow
S^1$  whose kernel is the metaplectic group. We give explicit formulas for
the multiplication in $\Mpc$ and for the representation of $\Mpc$ on the
spinor space, in terms of a nice parametrisation of the symplectic group as
described in \cite{refs:RobRaw} choosing a positive compatible complex
structure $j$ on $(V,\Omega)$. The subgroup $MU^c(V,\Omega,j)$ lying over
the unitary group $U(V,\Omega,j)$  is a trivial circle extension of it.

In Section \ref{mpc structures}, we recall what are $\Mpc$ structures on a
symplectic manifold; these always exist and are parametrized by Hermitean
complex line bundles over $M$. Using a positive compatible almost complex
structure $J$ on $(M,\Omega)$, we build explicitly those structures from
their restriction to the unitary frame bundle on $M$. Connections on $\Mpc$
structures are described in Section \ref{connections} and the symplectic
Dirac operator and its properties appear in Section \ref{Dirac}. We have
tried to give  a presentation which is self contained;  the content of  the
first sections is essentially taken from \cite{refs:RobRaw}. Although a
large part of this paper consists in putting together known results, we
believe that the point of view is new and opens some new possibilities in
symplectic geometry.

\section{The Heisenberg group and its Holomorphic Representation}
\label{section:Fock}

Let $(V,\Omega)$ be a finite-dimensional real symplectic vector space of
dimension $2n$. The \textit{Heisenberg group}  $H(V,\Omega)$ is the Lie group whose
underlying manifold is $V\times \R$ with multiplication
\[
(v_1,t_1)(v_2,t_2) = (v_1+v_2, t_1+t_2 - \half\Omega(v_1,v_2)).
\]
Its Lie algebra $\h(V,\Omega)$ has underlying vector
space $V\oplus\R$ with bracket
\[
[(v, \alpha), (w, \beta)] = (0, {}-\Omega(v,w))
\]
and is two-step nilpotent.
The exponential map is 
\[
\exp (v,\alpha) = (v,\alpha).
\]

In any (continuous) unitary irreducible representation $U$ of the
Heisenberg group on a separable Hilbert space $\H$, the centre
$\{0\}\times \R$ will act by multiples of the identity: $(0,t) \mapsto
e^{i\lambda t}I_{\H}$ for some real number $\lambda$ which we call the
\textit{central parameter}. When $\lambda=0$, $U$ arises from a
representation of the additive group of $V$, $U(v,t) =
e^{i\Omega(v,\mu)}$, on a one-dimensional space. For non-zero $\lambda$
the irreducible representation is infinite dimensional. It is known that
any two unitary irreducible representations with the same non-zero
central parameter are unitarily equivalent (Stone--von Neumann
Uniqueness Theorem) and one can change the value of $\lambda$ by
scaling: $U'(v,t) = U(cv,c^2t)$ has parameter $c^2\lambda$ if $U$ has
parameter $\lambda$, whilst the representation with parameter $-\lambda$
is the contragredient of that with $\lambda$. Thus up to scaling,
complex conjugation and equivalence there is just one infinite
dimensional unitary irreducible representation. It is fixed by
specifying its parameter which we take as $\lambda = -1/\hbar$ for some
positive real number $\hbar$.

This infinite dimensional representation can be constructed in a number
of ways, for example on $L^2(V/W)$ with $W$ a Lagrangian subspace of
$(V,\Omega)$ (Schr\"odinger picture). For our purposes it is most useful
to realise it on a Hilbert space of holomorphic functions (Fock picture
\cite{refs:Fock}) due in various forms to Segal \cite{refs:Segal}, Shale
\cite{refs:Shale}, Weil  \cite{refs:Weil} and Bargmann 
\cite{refs:Bargmann}. For this we consider the set of positive
compatible complex structures (PCCS) $j$ on $V$. More precisely, a
\textit{compatible complex structure} $j$ is a (real) linear map of $V$
which is symplectic, $\Omega(jv,jw) = \Omega(v,w)$, and satisfies
$j^2=-I_V$. These conditions on $j$ imply that $(v,w) \mapsto
\Omega(v,jw)$ is a non-degenerate symmetric bilinear form. We say $j$ is
\textit{positive} if this form is positive definite. Let 
$j_+(V,\Omega)$ denote the set of PCCS.

The compatible complex structures on $V$ form a finite number of orbits
under conjugation by elements of $Sp(V,\Omega)$ (see Section 
\ref{sec:symp+mpc}) which are distinguished
by the signature of the quadratic form $\Omega(v,jw)$. The stabilizer of
a point in the positive  orbit consists of those $g \in
Sp(V,\Omega)$ with $gjg^{-1} = j$ or equivalently $gj=jg$ so they are
complex linear for the complex vector space structure on $V$ defined by
$j$: $(x+iy)v = xv + yj(v)$. They also preserve $\Omega(v,jw)$ and hence 
preserve the Hermitean structure
\[
\langle v, w\rangle_j = \Omega(v,jw) - i \Omega(v,w), \qquad |v|_j^2 
 = \langle v,v\rangle_j.
\]
Thus picking $j \in j_+(V,\Omega)$ gives a complex Hilbert space
$(V,\Omega,j)$ of complex dimension $n= \frac12 \dim_{\R} V$. The
stabilizer of $j$ is the unitary group $U(V,\Omega,j)$ of this Hilbert space.

If we put $h=2\pi\hbar$ we may consider the Hilbert space $\H(V,\Omega,j)$ of
holomorphic functions $f(z)$ on $(V,\Omega,j)$ which are $L^2$ in the sense
that the norm $\|f\|_j$ given by
\[
\|f\|_j^2 = h^{-n} \int_V |f(z)|^2 e^{-\frac{|z|_j^2}{2\hbar}}dz
\]
is finite where $dz$ denotes the normalised Lebesgue volume on $V$ for the
norm \hbox{$|\cdot|_j$}. This holomorphic function Hilbert space has a
reproducing kernel or family of coherent states $e_v$ parametrised by $V$
given by
\[
(e_v)(z) = e^{\frac1{2\hbar}\langle z,v\rangle_j} 
\]
such that
\[
f(z) = (f, e_z)_j
\]
where $(f_1,f_2)_j$ is the inner product in $\H(V,\Omega,j)$ giving the
norm $\|f\|_j$.

$H(V,\Omega)$ acts unitarily and irreducibly on $\H(V,\Omega,j)$ by
\[
(U_j(v,t)f)(z) = e^{-it/\hbar + \langle z,v\rangle_j/2\hbar
- |v|_j^2/4\hbar}f(z-v).
\]

The Heisenberg Lie algebra $\h(V,\Omega)$ then has a skew-Hermitean
representation on smooth vectors $\H(V,\Omega,j)^\infty$ of this
representation (and these vectors include the coherent states $e_v$). If
$f\in\H(V,\Omega,j)^\infty$ we have
\[
(\dot U_j(v,\alpha) f)(z) 
= -i\alpha/\hbar f(z) + \frac1{2\hbar}\langle z,v\rangle_j f(z) 
- (\partial_zf)(v).
\]
If we put $cl(v) = \dot U_j(v,0)$ we get operators on the smooth vectors
$\H(V,\Omega,j)^\infty$ in $\H(V,\Omega,j)$ by
\[
(cl(v)f)(z) = \frac1{2\hbar}\langle z,v\rangle_j f(z) - (\partial_zf)(v)
\]
which are called Clifford multiplication. They satisfy
\[
cl(v)(cl(w) f) - cl(w)(cl(v) f) = \frac{i}{\hbar} \Omega(v,w) f.
\]
If we extend the representation of $\h(V,\Omega)$ to its enveloping
algebra then $\H(V,\Omega,j)^\infty$ becomes a Fr\'{e}chet space with
seminorms $f \mapsto \|u\cdot f\|_j$ for $u$ in the enveloping algebra,
and its dual $\H(V,\Omega,j)^{-\infty}$ can be viewed as containing
$\H(V,\Omega,j)$ so we have a Gelfand triple $\H(V,\Omega,j)^\infty
\subset \H(V,\Omega,j) \subset \H(V,\Omega,j)^{-\infty}$ on which the
enveloping algebra acts compatibly.

It is convenient to write the Clifford multiplication 
in terms of creation and annihilation operators:

\begin{definition}
For $v \in V$ define operators $c(v),a(v)$ on $\H(V,\Omega,j)$ by
\begin{equation}
(c(v)f)(z) = \frac1{2\hbar} \langle z,v \rangle_j f(z),
\qquad
(a(v)f)(z) = (\partial_z f)(v),\qquad f \in \H(V,\Omega,j). 
\end{equation}
$c(v)$ is called the \textbf{creation operator} in the direction $v$ and $a(v)$
the \textbf{annihilation operator}.
\end{definition}

Note, in this definition $f$ is initially taken in the smooth vectors
$\H(V,\Omega,j)^\infty$ but can be viewed as acting on $\H(V,\Omega,j)$
or $\H(V,\Omega,j)^{-\infty}$ in the distributional sense.

Polynomials in $z$ form a dense subspace of $\H(V,\Omega,j)$. The operator $c(v)$
acting on a polynomial multiplies it by a linear form so increases its
degree by $1$ whilst $a(v)$ performs a directional derivative so reduces
it by $1$. Easy calculations show:

\begin{proposition}\label{prop:create}
The creation and annihilation operators satisfy
\[
cl(v) = c(v) - a(v));\quad
a(v) = c(v)^*;\quad a(jv) \, =\,  i a(v);   \quad  c(jv) \, =\,  -i c(v); \]
\[[a(u),c(v)] = \frac1{2\hbar} \langle u,v \rangle_j ; \quad
[a(u),a(v)] \, =\,  0; \quad
[c(u),c(v)] \, = \,  0.
\]
\end{proposition}

Annihilation and creation operators are essentially the splitting of
Clifford multiplication into $j$-linear and $j$-anti-linear parts.

One advantage of using this holomorphic realisation of the basic
representation of the Heisenberg group is that the action of linear
operators on $\H(V,\Omega,j)$ is not only determined by what they do to the
coherent states (since the span of the latter is dense), but there is a
formula for the operator on general vectors in terms of an integral kernel
constructed from the coherent states. This works just as well for unbounded
operators so long as they (and their formal adjoint) are defined on the
coherent states.

\begin{definition}\label{def:Berezinkernel}
Let $A$ be a linear operator defined, along with its formal adjoint, on
a dense domain in $\H(V,\Omega,j)$ containing the coherent states.
We define its \textbf{Berezin kernel} \cite{refs:Berezin} to be 
\[
A(z,w) = (Ae_w)(z) = (Ae_w,e_z)_j = (e_w,A^*e_z)_j = \ol{(A^*e_z)(w)}\,.
\]
\end{definition}

\begin{proposition}
Let $A$ be a linear operator and $A(z,w)$ be its Berezin kernel as in
Definition \ref{def:Berezinkernel}. Then $A(z,w)$ is holomorphic in $z$
and anti-holomorphic in $w$. For any $f$ in the domain of $A$ and $A^*$
we have
\[
(Af)(z) = h^{-n}\int_V f(w) A(z,w) e^{-\frac{|w|^2_j}{2\hbar}}\,dw.
\]
\end{proposition}

\begin{example}
The Berezin kernel of the identity map $I_{\H}$ of $\H(V,\Omega,j)$ is
$I_{\H}(z,w) = \exp \frac{1}{2\hbar} \langle z,w \rangle_j$, the coherent
states themselves.
\end{example}

\section{The symplectic and $\Mpc$ groups and their Lie algebras}
\label{sec:symp+mpc}

Let $(V,\Omega)$ be a symplectic vector space. We denote by 
$Sp(V,\Omega)$ the Lie group of invertible linear maps $g \colon V \to V$
such that $\Omega(gv,gw) = \Omega(v,w)$ for all $v,w \in V$. Its Lie algebra
$\sp(V,\Omega)$ consists of linear maps $\xi \colon V \to V$ with
$\Omega(\xi v,w)  + \Omega(v,\xi w)  = 0$ for all $v,w \in V$ or equivalently
$(u,v) \mapsto \Omega(u,\xi v)$ is a symmetric bilinear form.

$Sp(V,\Omega)$ acts as a group of automorphisms of the Heisenberg group $H(V,\Omega)$ by
\[
g\cdot(v,t) = (g(v),t).
\]

By composing the representation $U_j$ of $H(V,\Omega)$ on $\H(V,\Omega,j)$
with an automorphism $g \in Sp(V,\Omega)$ we get a second representation of
$H(V,\Omega)$ also on $\H(V,\Omega,j)$:
\[
U^g_j(v,t) = U_j(g\cdot(v,t)) = U_j(g(v),t)
\]
and evidently the representation $U^g_j$ is still irreducible and has the same
central parameter $-\frac1{\hbar}$. By the Stone--von~Neumann Uniqueness
Theorem there is a unitary transformation $U$ of $\H(V,\Omega,j)$ such that
\begin{equation}\label{mpc:def}
U^g_j= U\,U_j\,U^{-1}.
\end{equation}
Since $U_j$ is irreducible the operator $U$ is determined up to a
scalar multiple by the corresponding elements $g$ of $Sp(V,\Omega)$, and
it is known to be impossible to make a continuous choice $U_g$ which
respects the group multiplication. Instead we view the operators $U$ as
forming a new group.
\begin{definition}\label{defMpc}
The group  $\Mpc(V,\Omega,j)$ consist of the pairs
$(U,g)$ of unitary transformations $U$ of $\H(V,\Omega,j)$ and elements
$g$ of $Sp(V,\Omega)$ satisfying (\ref{mpc:def}). The multiplication law
in $\Mpc(V,\Omega,j)$ is diagonal.
\end{definition}
The map
\[
\sigma(U,g) = g
\] 
is a surjective homomorphism to $Sp(V,\Omega)$ with kernel of $\sigma$
consisting of all unitary multiples of the identity, so we have a central
extension
\begin{equation}\label{mpc:ext}
1 \map U(1) \map \Mpc(V,\Omega,j) \map[\sigma] Sp(V,\Omega) \map 1
\end{equation}
which does not split.

\subsection{Parametrising the symplectic group} 

We now describe a useful parametrisation of the real
symplectic group, depending  on the triple $(V,\Omega,j)$
[so we have now chosen and fixed
$j \in j_+(V,\Omega)$].
This description is compatible with the Fock representation and
will allow an explicit description of the group $\Mpc(V,\Omega,j)$. 
Consider $GL (V,j)=\setdef{g \in GL(V)}{gj = jg}$ and observe that
$U(V,\Omega,j) = Sp(V,\Omega) \cap GL(V,j)$.
[Viewing $V$ as a complex vector space using $j$, then $U(V, \Omega , j)$ is
isomorphic  to $U (n)$ and $GL (V,j)$ to $GL (n,\C )$.]

Any  $g \in Sp (V, \Omega )$ decomposes \emph{uniquely} as a sum of a
$j$-linear and $j$-antilinear part,
\[
g = C_g +D_g, 
\]
where $C_g = \onehalf (g - jgj)$ and $D_g = \onehalf (g + jgj)$. It is
immediate that for any non-zero $v \in V$ 
\[
 4\Omega (C_gv, jC_gv) \, = \, 2\Omega (v,jv) + \Omega (gv, jgv) +
\Omega (gjv, jgjv) > 0
\]
so that $C_g$ is invertible and we have :
\begin{lemma}\label{Lemma3}
$C_g \in GL (V,j)$ for all $g \in Sp(V,\Omega)$. If $g \in U (V,
\Omega , j)$ we have $C_g =g$.
\end{lemma}

Set $Z_g = C^{-1}_gD_g$, then $g = C_g(1+Z_g)$ and $Z_g$ is $\C $-antilinear.

Write $g^{-1} = C_{g^{-1}}(1+Z_{g^{-1}})$. Equating $\C$-linear 
and $\C$-antilinear parts in $1= g^{-1}g $ gives
$1 = C_{g^{-1}}(C_g + Z_{g^{-1}}C_gZ_g)$ and  $
0 = C_{g^{-1}}(Z_{g^{-1}}C_g +C_gZ_g)$, hence 
 
\[
Z_{g^{-1}} = -C_gZ_gC^{-1}_g \text{\ \ and\ \ } 1=
C_{g^{-1}}C_g(1-Z^2_g).
\]
 Thus $1-Z^2_g$ is 
invertible with $(1-Z^2_g)^{-1} = C_{g^{-1}}C_g$. We 
can also decompose a product 
\[
C_{g_1g_2}(1+Z_{g_1g_2}) 
= C_{g_1}(1+Z_{g_1})C_{g_2}(1+Z_{g_2})
\]
into linear and anti-linear parts 
\begin{equation}\label{eqn:Cproduct}
C_{g_1g_2} = C_{g_1}(C_{g_2}+Z_{g_1}C_{g_2}Z_{g_2}) 
\, = \, 
C_{g_1}(1-Z_{g_1}Z_{g^{-1}_2})C_{g_2} 
\end{equation}
and
\begin{eqnarray*}
Z_{g_1g_2} 
&= & C^{-1}_{g_2}(1-Z_{g_1}Z_{g^{-1}_2})^{-1}(Z_{g_1}C_{g_2} 
+ C_{g_2}Z_{g_2}) \\ 
& = & C^{-1}_{g_2}(1-Z_{g_1}Z_{g^{-1}_2})^{-1}
(Z_{g_1}- Z_{g^{-1}_2})C_{g_2}
\end{eqnarray*}
allowing us to write the group structure in terms of the $C$ and $Z$ parameters.

We claim that the function 
\[
(g^{}_1, g^{}_2) \mapsto \det (1 - Z_{g^{}_1}Z_{g^{-1}_2})
\]
has a smooth logarithm. To see this we determine that the set where $Z_g$ 
lives is  the Siegel domain.  

Since  $\langle u, v \rangle_j=\Omega(u,jv)-i\Omega(u,v)$, and 
$\Omega(gu,v)=\Omega(u,g^{-1}v)$ we have
\begin{eqnarray*}
 \langle  gu,v \rangle_j & = &\Omega(u,j(-jg^{-1}jv)-i\Omega(u,g^{-1}v) \\ 
  \langle  jgju,v \rangle_j & = &\Omega(u,j(-g^{-1}v))-i\Omega(u,jg^{-1}jv).\, 
  \end{eqnarray*}
Subtracting and adding the two relations above give
\begin{eqnarray*}
2\, \langle C_gu,v \rangle_j 
& =& \Omega (u,j(g^{-1}-jg^{-1}j)v )-i\Omega (u,(g^{-1}-jg^{-1}j)v) \\ 
& =& \langle u, (g^{-1}-jg^{-1}j)v \rangle_j \,
= \,  2\, \langle u,C_{g^{-1}}v \rangle_j \\
2\, \langle C_gZ_gu,v \rangle_j 
& =& -\Omega (u,j(g^{-1}+jg^{-1}j)v )-i\Omega (u,(g^{-1}+jg^{-1}j)v) \\ 
& =&- {\overline{ \langle u, (g^{-1}+jg^{-1}j)v \rangle_j }} 
\, =\,  -2\, \langle C_{g^{-1}}Z_{g^{-1}}v,u \rangle_j .
\end{eqnarray*}
Therefore 
\[
C^*_g = C_{g^{-1}}.
\]
Moreover, 
\[
1-Z^2_g = (C_{g^{-1}}C_g)^{-1} \, = \, (C^*_gC_g)^{-1}, 
\]
which is positive definite. Since $Z_g$ is antilinear and 
$\langle v, w \rangle_j $ is antilinear in $w$, the function $(v,w) \mapsto 
\langle v, Z_gw \rangle_j $ is complex bilinear. We have
\[
 \langle Z_gu,v \rangle_j =\langle C_gZ_gu, C_g^{*-1}v \rangle_j 
 =\langle C_gZ_gu, C_{g^{-1}}^{-1}v \rangle_j 
 =- \langle C_{g^{-1}}Z_{g^{-1}} C_{g^{-1}}^{-1}v,u \rangle_j
 = \langle Z_gv,u \rangle_j .
\]
Hence
\begin{lemma}\label{Lemma4}
$Z_g$ has the three properties: $Z_g$ is $\C $-antilinear; $(v,w) \mapsto
\langle v, Z_gw \rangle_j $ is symmetric; $1-Z^2_g$ is self adjoint and
positive definite.
\end{lemma}
Let $\B(V, \Omega , j)$ be the \emph{Siegel domain} consisting of 
$Z \in \mathrm{End}(V)$ such that 
\[
Zj = -jZ, \, \,\langle v, Zw \rangle_j = \langle w, Zv \rangle_j ,
\text{\ and\ }
1-Z^2 \, \text{\ is positive definite.} 
\]

\begin{theorem}
We have an injective map 
\[
Sp (V, \Omega ) \rightarrow GL (V, j) \times \B(V, \Omega , j): 
g \mapsto (C_g, Z_g), 
\]
whose image is the set $\setdef{(C,Z)}{1-Z^2 = (C^*C)^{-1}}$.
\end{theorem}
Indeed, for any such $(C,Z)$, define $g=C(1+Z)$. We have
\begin{eqnarray*}
\Omega(gu,gv)&=&-\Im \langle C(1+Z)u,C(1+Z)v\rangle_j
=-\Im \langle (1+Z)u,C^*C(1+Z)v\rangle_j\\
&=&\Omega((1+Z)u,(1-Z)^{-1}v)=\Omega(u,(1-Z)^{-1}v)-\Omega(u,Z(1-Z)^{-1}v)\\
&=&\Omega(u,v).
\end{eqnarray*}
In order to parametrise $\Mpc(V,\Omega,j)$ in a similar fashion we observe
\begin{proposition}\label{Prop1}
If $Z_1,Z_2 \in \B(V,\Omega , j)$, then $1-Z_1Z_2 \in GL(V,j)$ and its
real part, $\half ((1-Z_1Z_2) + (1-Z_1Z_2)^*) = \half ((1-Z_1Z_2) +
(1-Z_2Z_1))$, is positive definite.
\end{proposition}
Indeed, one has $\langle Z_2Z_1 u,v\rangle_j=\langle Z_2v, Z_1 u\rangle_j
=\langle u, Z_1Z_2v\rangle_j$ and 
\begin{eqnarray*}
\langle(1- Z_1Z_2) v+(1- Z_2Z_1) v,v\rangle_j&&\\
&&\mbox{}\kern-1cm\mbox{}=\langle(1- Z_1)^2 v,v\rangle_j
+\langle (1-Z_2^2)v, v\rangle_j+\langle (Z_1-Z_2)^2 v,v\rangle_j \\
&&\mbox{}\kern-1cm\mbox{}=\langle(1- Z_1^2) v,v\rangle_j
+\langle(1-Z_2)^2 v,v\rangle_j +\Vert (Z_1-Z_2)v\Vert_j^2.
\end{eqnarray*}
Thus $1-Z_1Z_2 \in GL(V, j)_+$ where
\[
GL (V, j)_+ = \setdef{g \in GL(V,j)}{g+g^* \mbox{ is positive definite}}.
\] 
Any $g \in GL (V, j)$ can be written uniquely in the form $X+iY$ with $X$
and $Y$ self-adjoint, and $g \in GL (V, j)_+$ when $X$ is positive
definite. Positive definite self adjoint operators $X$ are of the form
$X=e^Z$ with $Z$ self-adjoint and $Z \mapsto e^Z$ is a diffeomorphism of
all self-adjoint operators with those which are positive definite. Given
self-adjoint operators $X$ and $Y$ with $X$ positive definite then $X+iY$
has no kernel, so is in $GL (V, j)$. Thus $GL (V, j)_+$ is an open set in
$GL (V, j)$ diffeomorphic to the product of two copies of the real vector
space of Hermitean linear maps of $(V,\Omega,j)$. In particular $GL (V,
j)_+$ is simply-connected. Thus  there is a unique smooth function $a
\colon GL (V, j)_+ \to \C$ such that
\begin{equation}
\det g = e^{a(g)}, \qquad g \in GL (V, j)_+
\end{equation}
and normalised by $a(I)=0$. Further, since $\det$ is a holomorphic function on
$GL(V, j)$, $a$ will be holomorphic on $GL (V, j)_+$. Hence
$\det(1-Z_1Z_2) = e^{a(1-Z_1Z_2)}$ and is holomorphic in $Z_1$ and
anti-holomorphic in $Z_2$.

\subsection{Parametrising the $\Mpc$ group}

We defined the group $\Mpc(V,\Omega,j)$ as pairs $(U,g)$ with $U$ a
unitary operator on $\H(V,\Omega,j)$ and $g \in Sp(V,\Omega)$ satisfying
(\ref{mpc:def}). Here we determine the form of the operator $U$ in terms
of the parameters $C_g, Z_g$ of $g$ introduced in the previous
paragraph. Fixing a $j \in j_+(V,\Omega)$, any bounded operator $A$ on
$\H(V,\Omega,j)$ is determined by its Berezin kernel  $A(z,v)=\left(
Ae_v, e_z\right)_j=(Ae_v)(z)$ as in Definition \ref{def:Berezinkernel}.

Since $(e_v)(z) = e^{\frac1{2\hbar}\langle z,v\rangle_j} =
e^{\frac1{4\hbar}\langle v,v\rangle_j}(U_j(v,0)e_0)(z)$, the kernel
$U(z,w)$ of an operator $U$ with $(U,g)\in \Mpc(V,\Omega,j)$ is such that
\[ \
U(z,v)=e^{\frac1{4\hbar}(\langle z,z\rangle_j+\langle v,v\rangle_j)}
\left( U U_j(v,0)e_0,U_j(z,0)e_0\right)_j;
\]
so that, if $v=g^{-1}z$, since $U U_j(g^{-1}z,0)=U_j(z,0)U$, we have
\[
U(z,g^{-1}z)=e^{\frac1{4\hbar}(\langle z,z\rangle_j+\langle g^{-1}z,g^{-1}z\rangle_j)}
\left( U e_0,e_0\right)_j.
\]
Now $U(z,v)$ is holomorphic in $z$ and antiholomorphic in $v$, so it is
completely determined by its values on $(z,v=g^{-1}z)$. When $v=g^{-1}z=
C_{g^{-1}}(1+Z_{g^{-1}})z$, one can write $z=C_{g^{-1}}^{-1}v-Z_{g^{-1}}z$
and $v=C_g^{-1}z-Z_gv$ so that
\[
U(z,v) = \left( U e_0,e_0\right)_j \exp \frac1{4\hbar}
\{\langle z,C_{g^{-1}}^{-1}v-Z_{g^{-1}}z\rangle_j
+\langle C_g^{-1}z-Z_gv,v\rangle_j\}.
\]
Hence
\begin{theorem}\label{thm:Uparameters}
If $(U,g) \in \Mpc(V,\Omega,j)$ then the Berezin kernel $U(z,v)$ of $U$
has the form
\begin{equation}\label{eqn:Ukernel}
U(z,v) = \lambda \exp \frac1{4\hbar}\{2\langle C_g^{-1}z,v\rangle_j
-\langle z, Z_{g^{-1}}z \rangle_j - \langle Z_g v,v\rangle_j \}
\end{equation}
for some $\lambda \in \C$ with $|\lambda^2 \Det C_g| =1$. Moreover 
$\lambda = 
(Ue_0)(0) = (Ue_0,e_0)_j$. 
\end{theorem}
The fact that $|\lambda^2 \Det C_g| =1$ comes from 
\begin{eqnarray*}
\Vert Ue_0\Vert^2_j=\Vert e_0\Vert^2_j=1
&=&h^{-n}\int_V U(z,0)\overline{U(z,0)}e^{-\frac{\vert z\vert^2_j}{2\hbar}}dz\\
&=&\vert\lambda\vert^2h^{-n}\int_Ve^{-\frac1{4\hbar}(  \langle z,Z_{g^{-1}}z\rangle_j 
+ \langle Z_{g^{-1}}z,z\rangle_j )} e^{-\frac{\vert z\vert^2_j}{2\hbar}}dz\\
&=& \vert\lambda\vert^2 \det(1-Z_{g^{-1}}^2)^{-\half}=\vert\lambda\vert^2 \det(C_gC_g^*)^\half=
|\lambda^2 \Det C_g| .
\end{eqnarray*}
We have used the classical result for Gaussian integrals (see for instance \cite{refs:Folland}):
\[
\int_Ve^{-\frac{\pi}{2}(  \langle z,Z_1z\rangle_j + \langle Z_2 z,z\rangle_j )} 
e^{-\pi{\vert z\vert^2_j}}dz= \det(1-Z_1Z_2)^{-\half}=e^{-\half a(1-Z_1Z_2)}
\]
We call $g,\lambda$ given by Theorem \ref{thm:Uparameters} the
\textit{parameters} of $U$. To write the
multiplication in $\Mpc(V,\Omega,j)$  in terms of those parameters, we observe
that
\begin{eqnarray*}
\left((U_1U_2)e_0,e_0\right)_j&=&\left(U_2e_0,U_1^*e_0\right)_j
=h^{-n}\int_V(U_2e_0)(z)\overline{(U_1^*e_0)(z)}e^{-\frac{\vert z\vert^2_j}{2\hbar}}dw\\
&=&h^{-n}\int_V U_1(0,z)U_2(z,0)e^{-\frac{\vert z\vert^2_j}{2\hbar}}dz\\
&=&\lambda_1\lambda_2 h^{-n}\int_Ve^{-\frac1{4\hbar}(  \langle z,Z_{g_2^{-1}}z\rangle_j 
+ \langle Z_{g_1}z,z\rangle_j )} e^{-\frac{\vert z\vert^2_j}{2\hbar}}dz\\
&=& \lambda_1\lambda_2  \det(1-Z_{g_1}Z_{g_2^{-1}})^{-\half}
\end{eqnarray*}
so that
\begin{theorem}\label{thm:mpcproduct}
The product in $\Mpc(V,\Omega,j)$ of $(U_i,g_i)$ with parameters
$g_i,\lambda_i$, $i = 1,2$ has parameters $g_1g_2, \lambda_{12}$ with
\[
\lambda_{12} 
=  \lambda_1\lambda_2 e^{-\frac12 a\left(1-Z_{g_1}Z_{g_2^{-1}}\right)}.
\]
\end{theorem}
\begin{corollary}\label{thm:eta}
The group $\Mpc(V,\Omega,j)$ is a Lie group. It admits a character $\eta$
 given by
\[
\eta(U,g) = \lambda^2 \Det C_g
\]
if $g, \lambda$ are the parameters of $U$. Its restriction to the central
$U(1)$ is the squaring map.

The inclusion $U(1) \hookrightarrow \Mpc(V,\Omega,j)$
sends $\lambda \in U(1)$ to $(\lambda I_{\H}, I_V)$ which has parameters
$I_V,\lambda$.
 \end{corollary}

\begin{definition}
The  \textit{metaplectic group} is the kernel of $\eta$; it is given by 
\[
Mp(V,\Omega,j) = \{(U,g)\in \Mpc(V,\Omega,j) \suchthat \lambda^2\Det C_g = 1\}
\]
with the multiplication rule given by Theorem \ref{thm:mpcproduct}. It is a
double covering of $Sp(V,\Omega)$.
\end{definition}
The sequence $1 \map U(1) \map \Mpc(V,\Omega,j) \map[\sigma] Sp(V,\Omega)
\map 1$ does not split, but  splits on the double covering
$Mp(V,\Omega,j)$.

 We let
$\MUc(V,\Omega,j)$ be the inverse image of $U(V,\Omega,j)$ under $\sigma$
so that (\ref{mpc:ext}) induces a corresponding short exact sequence
\begin{equation}\label{mpc:mucext}
1 \map U(1) \map \MUc(V,\Omega,j) \map[\sigma] U(V,\Omega,j) \map 1.
\end{equation}
$\MUc(V,\Omega,j)$ is  a maximal compact subgroup of
$\Mpc(V,\Omega)$.

\begin{proposition} \label{cor:muc}
If $(U,k) \in MU^c(V,\Omega,j)$ has parameters $k$ and $\lambda$ then
$\lambda(U,k) = \lambda$ is a character of $\MUc(V,\Omega,j)$. If $f \in
\H(V,\Omega,j)$ then $(Uf)(z) = \lambda f(k^{-1}z)$,
so unlike the exact sequence (\ref{mpc:ext}), (\ref{mpc:mucext}) does split
canonically by means of the homomorphism
\[
\lambda \colon \MUc(V,\Omega,j) \map U(1).
\]
This gives an isomorphism
\[
\MUc(V,\Omega,j) \map[\sigma\times\lambda] U(V,\Omega,j) \times U(1).
\]
\end{proposition}
In addition we have the determinant character $\Det \colon U(V,\Omega,j)
\map U(1)$ which can be composed with $\sigma$ to give a character
$\Det\circ\sigma$ of $\MUc(V,\Omega,j)$. The three characters $\eta$,
$\lambda$ and $\Det\circ\sigma$ are related by
\begin{equation}\label{mpc:chars}
\eta = \lambda^2 \Det\circ\sigma.
\end{equation}

Let $\mpc(V,\Omega,j)$ be the Lie algebra of $\Mpc(V,\Omega,j)$.
Differentiating (\ref{mpc:ext}) gives an exact sequence of Lie algebras
\begin{equation}\label{mpc:laext}
0 \map \u(1) \map \mpc(V,\Omega,j) \map[\sigma] \sp(V,\Omega) \map 0.
\end{equation}
We denote by $\eta_* \colon \mpc(V,\Omega,j) \map \u(1)$  the differential
of the group homomorphism $\eta$, and observe that $\half \eta_*$ is a map to
$\u(1)$ which is the identity on the central $\u(1)$ of $\mpc(V,\Omega,j)$.
Hence (\ref{mpc:laext}) splits as a sequence of Lie algebras. We shall
refer to the component in $\u(1)$ of an element $\xi$ of $\mpc(V,\Omega,j)$
as its \textit{central component} $\xi^c$.

To end this section we describe the Lie algebra representation of
$\mpc(V,\Omega,j)$ on (smooth vectors of) $\H(V,\Omega,j)$. Elements of
$\mpc(V,\Omega,j)$ are given by pairs $(\mu, \xi)$ where $\mu\in \u(1)$ and
$\xi\in \sp(V,\Omega)$ and
\[
[(\mu_1,\xi_1),(\mu_2,\xi_2)] = (0, \xi_1\xi_2 - \xi_2\xi_1).
\]
Given $(\mu,\xi) \in \mpc(V,\Omega,j)$, let $(U_t,g_t)$ be a curve in
$\Mpc(V,\Omega,j)$ with parameters $(\lambda_t,g_t)$ passing through the
identity, $(U_0,g_0) = (I_{\H},I_V)$, and with tangent $(\mu,\xi)$ so that
${\dot\lambda}_0 = \mu$ and ${\dot g}_0 = \xi$. Then if we put $C_t =
C_{g_t}$ and $Z_t = Z_{g_t}$ we have $\xi = {\dot C}_0 + {\dot Z}_0 = \eta
+ \zeta$ with ${\dot C}_0 = \half (\xi - j\xi j) = \eta$ and ${\dot Z}_0 =
\half (\xi + j\xi j) = \zeta$.

We now differentiate equation (\ref{eqn:Ukernel}) at $t=0$ 
\[
{\dot U}_0(z,v) = \left\{ \mu  
-\frac{1}{2\hbar}\langle \eta z,v\rangle_j 
+\frac{1}{4\hbar}\langle z,\zeta z\rangle_j
-\frac{1}{4\hbar}\langle \zeta v,v \rangle_j
\right\} \exp \frac{1}{2\hbar}\langle z,v \rangle_j
\]
which is a kernel for the action of $(\mu,\xi)$ on smooth vectors of
$\H(V,\Omega,j)$, or on all of $\H(V,\Omega,j)$ if we view the result as in
$\H(V,\Omega,j)^{-\infty}$. The final step is to convert this to an action
on a general smooth vector $f(z)$.

Obviously $\mu$ acts as a multiple of the identity. Let $2n=\dim_{\R} V$. For a
$j$-linear map $\eta$ of $V$, the operator with kernel
$\frac{1}{2\hbar}\langle \eta z, v \rangle_j  \exp \frac{1}{2\hbar}\langle
z,v \rangle_j$ is given by $f \mapsto f_1$ where \begin{eqnarray*} f_1(z)
&=& h^{-n}\int_V f(v) \frac{1}{2\hbar}\langle \eta z, v \rangle_j \exp
\frac{1}{2\hbar}\langle z,v \rangle_j \exp -\frac{|v|^2_j}{2\hbar} \,dv\\
&=& \partial_z\left( h^{-n}\int_V f(v) \exp \frac{1}{2\hbar}\langle z,v
\rangle_j \exp -\frac{|v|^2_j}{2\hbar} \,dv\right)(\eta z)=
(\partial_z f) (\eta z). 
\end{eqnarray*}

Clearly the kernel $\frac{1}{4\hbar}\langle z,\zeta z\rangle_j \exp
\frac{1}{2\hbar}\langle x,v\rangle_j$ defines an operator $f \mapsto f_2$
where
\[
f_2(z) = \frac{1}{4\hbar}\langle z,\zeta z\rangle_j f(z).
\]
Finally, the kernel $\frac{1}{4\hbar}\langle \zeta v,v \rangle_j \exp
\frac{1}{2\hbar}\langle z,v \rangle_j$ defines an operator $f \mapsto f_3$
where
\begin{eqnarray*}
f_3(z)
&=& h^{-n}\int_V f(v) \frac{1}{4\hbar}\langle \zeta v,v \rangle_j 
\exp \frac{1}{2\hbar}\langle z,v \rangle_j e^{- \frac{|v|^2_j}{2\hbar}}\, dv\\
&=& h^{-n}\int_V f(v)\frac12 ( \partial_z \exp 
\frac{1}{2\hbar}\langle z,v\rangle_j)(\zeta v)  e^{- \frac{|v|^2_j}{2\hbar}}\, dv.
\end{eqnarray*}
Take an orthonormal basis $e_1,\dots, e_n$ for $V$ as a complex Hilbert
space then
\[
( \partial_z \exp \frac{1}{2\hbar}\langle z,v\rangle_j)(\zeta v)
= \sum_{i=1}^n ( \partial_z \exp \frac{1}{2\hbar}\langle z,v\rangle_j)(\zeta e_i)
\langle e_i, v \rangle_j
\]
so
\begin{eqnarray*}
f_3(z) &=& \sum_{i=1}^n  \partial_z \left(
h^{-n}\int_V f(v)\frac12 \langle e_i,v\rangle_j
\exp \frac{1}{2\hbar}\langle z,v\rangle_j 
e^{- \frac{|v|^2_j}{2\hbar}}\, dv \right)(\zeta e_i)\\
 &=& \hbar \sum_{i=1}^n  \partial_z \left(
h^{-n}\int_V f(v) \partial_z(\exp 
\frac{1}{2\hbar}\langle z,v\rangle_j)(e_i) 
e^{- \frac{|v|^2_j}{2\hbar}}\, dv \right)(\zeta e_i)\\
&=&  \hbar \sum_{i=1}^n  \partial_z \left( 
\partial_z f (e_i)\right)(\zeta e_i)
\end{eqnarray*}
Hence we have shown
\begin{proposition}
If $(\mu,\xi) \in \mpc(V,\Omega,j)$ and $\xi = \eta +\zeta$ is the
decomposition of $\xi$ into the part, $\eta$, which commutes with $j$ and
the part, $\zeta$, which anti-commutes then the action of $(\mu,\xi)$ on
spinors is given by
\[
((\mu, \xi)f)(z) = \mu f - (\partial_z f) (\eta z) 
+ \frac{1}{4\hbar}\langle z,\zeta z\rangle_j f(z)
- \hbar \sum_{i=1}^n  \partial_z \left( \partial_z f (e_i)\right)(\zeta e_i).
\]
\end{proposition}

\section{$\Mpc$ structures}\label{mpc structures}

\begin{definition}
Fix a symplectic vector bundle $(E,\omega)$ of rank $2n$ over $M$ and let
$(V,\Omega)$ be a fixed symplectic vector space of dimension $2n$. Then the
\textit{symplectic frame bundle} \hbox{$\pi \colon Sp(E,\omega) \map M$} of
$(E,\omega)$ is the bundle whose fibre at $x \in M$ consists of all
symplectic isomorphisms $b \colon V \map E_x$. By composition on the right
it becomes a principal $Sp(V,\Omega)$ bundle.
\end{definition}

\begin{definition}
By an \textit{$\Mpc$ structure} on $(E,\omega)$ we mean a principal
$\Mpc(V,\Omega,j)$ bundle $ \pi^P \colon P \map M$ with a fibre-preserving
map $\phi \colon P \map Sp(E,\omega)$ such that the group actions are
compatible:
\[
\phi(p\cdot g) = \phi(p)\cdot \sigma(g), \quad 
\forall p \in P,\ g \in \Mpc(V,\Omega,j).
\]
\end{definition}

We shall use the following notation: If $\pi^P \colon P \to M$ is a
principal $G$-bundle and $\fhi \colon G \to H$ a homomorphism of Lie
groups, then $P\times_G H:=(P\times H)_{\sim_G}$ denotes the bundle
whose elements are equivalence classes of elements in $P\times H$ for
the equivalence defined by $G$, i.e. $(\xi,h)\sim_G(\xi\cdot
g,\fhi(g^{-1})h)$ for any $g\in G$.  The  bundle $P\times_G H$ has a
right free $H$ action making it into a principal $H$-bundle. Equivalence
classes will be denoted by square brackets. Given any $U(1)$-bundle
$L^{(1)}$, we denote by $L$  the associated Hermitean complex line
bundle $L=L^{(1)}\times_{U(1)}\C$ and reciprocally, given any Hermitean
line bundle $L$ over $M$ we denote by $L^{(1)} = \{ q \in L \suchthat
|q|=1\}$ the corresponding $U(1)$-bundle. Given two fibre bundles over
$M$, $\pi^F \colon F \to M$ and $\pi^K \colon K \to M$, we denote by
$F\times_MK$ the fibre-wise product bundle $F\times_MK:=\{\, (\xi,\xi')
\in F\times G\,\vert \, \pi^F(\xi)=\pi^K(\xi')\,\}$.

\begin{theorem}
Every symplectic vector bundle $(E,\omega)$ admits an $\Mpc$ structure, and
the isomorphism classes of $\Mpc$ structures on $(E,\omega)$ are
parametrized by line bundles. Choosing  a fibre-wise positive
$\omega$-compatible complex structure $J$ on $E$, those symplectic frames
which are also complex linear form a principal $U(V,\Omega,j)$ bundle
called the unitary frame bundle which we denote by $U(E,\omega,J)$. If $P$
is an $\Mpc$ structure, we  look at the subset $P_J$ of $P$ lying over the
unitary frames
\[
P_J:=\phi^{-1}(U(E,\omega,J)).
\]
This will be a principal $\MUc(V,\Omega,j)\simeq_{\sigma\times\lambda}
U(V,\Omega,j) \times U(1)$ bundle. Clearly
\[
P\simeq P_J\times_{\MUc(V,\Omega,j)}\Mpc(V,\Omega,j).
\]
We denote by $P^{(1)}_J(\lambda):=P_J\times_{\MUc(V,\Omega,j),\lambda}U(1)$
the $U(1)$ principal bundle associated to $P_J$ by the homomorphism
$\lambda$. We have a map $\tilde\lambda: P_J\rightarrow P^{(1)}_J(\lambda)
: \xi\mapsto [(\xi,1)]$. Then
\[
{\phi\times \tilde\lambda}: P_J\rightarrow U(E,\omega,J)\times_M
P^{(1)}_J(\lambda) : \xi \mapsto \phi(\xi),[(\xi,1)]
\]
is an isomorphism, with the right action of $\MUc(V,\Omega,j)$ on the
right-hand side given via $ \sigma\times\lambda $ by the action of
$U(V,\Omega,j)$ on the $U(E,\omega,J)$  and of $U(1)$ on
$P^{(1)}_J(\lambda)$.\\ The line bundle associated to $P_J$ by the
character $\lambda$ is denoted  by $P_J(\lambda)$; its isomorphism class is
independent of the choice of $J$.

Reciprocally, given any Hermitean line bundle $L$ over $M$
\[
(U(E,\omega,J)\times_M L^{(1)})\times_{\MUc(V,\Omega,j)}\Mpc(V,\Omega,j)
\]
has an obvious structure of $\Mpc$ structure over $M$.
\end{theorem} 
Remark that choosing a $J$ is always possible as the bundle of fibrewise
positive $\omega$-compatible complex structures has contractible fibres;
any two such positive almost complex structures are homotopic. For a given
choice of such a $J$ we denote by $g_J$ the corresponding Riemannian metric
on $M$ : $g_J(X,Y):=\omega(X,JY)$.

The map which sends the isomorphism class of $P$ to the isomorphism class
of $P_J(\lambda)$ is the required parametrisation of $\Mpc$ structures by
line bundles.

If $P$ is an $\Mpc$ structure on $(E,\omega),$ let
$P^{(1)}(\eta)=P\times_{\Mpc,\eta} U(1)$ denote the U(1) bundle associated
to $P$ by the homomorphism $\eta$ of $\Mpc(V,\Omega,j)$. We denote by
$\tilde\eta$ the map
\[
\tilde\eta: P\rightarrow P^{(1)}(\eta) : \xi \mapsto [(\xi,1)].
\]
The total Chern class of the unitary structure $ U(E,\omega,J)$ depends only on
$(E,\omega)$ and not on the choice of $J$. We denote by $c_1(E,\omega)$ the
first Chern class of the unitary structure determined by picking any positive
$\omega$-compatible complex structure in $E$. Note that $c_1(E,\omega)$ is
the first Chern class of the line bundle associated to the determinant
character $\Det$ of $U(V,\Omega,j)$.

From the relationship (\ref{mpc:chars}) between the three characters we
have
\[
c_1(P(\eta)) = 2c_1(P_J(\lambda)) + c_1(E,\omega).
\]

Let $\pi^{L^{(1)}} \colon L^{(1)}\to M$ be a principal $U(1)$ bundle over
$M$ and $\pi^c \colon P \to M$ an $\Mpc$-structure then the fibre product 
$L^{(1)} \times_M P$ is obviously a $U(1) \times \Mpc(V,\Omega,j)$ bundle
over $M$ and the associated bundle with fibre $\Mpc(V,\Omega,j)$ to the
homomorphism
\begin{equation}\label{eqn:twistgroups}
(z,(U,g)) \mapsto (zU,g) \quad\colon\quad U(1) \times \Mpc(V,\Omega,j) \to
\Mpc(V,\Omega,j),
\end{equation}
gives a principal $\Mpc(V,\Omega,j)$ bundle $(L^{(1)}\times_M
P)\times_{U(1) \times \Mpc(V,\Omega,j)} \Mpc(V,\Omega,j)$. \\ We denote it
by $L^{(1)}\cdot P$ or by $L\cdot P$.

Note that $(L^{(1)}\cdot P)(\eta) = L^2 \otimes P(\eta)$. We also have
$(L^{(1)}\cdot P)_J(\lambda)=L\otimes P_J(\lambda)$. This gives  a
simply-transitive action of the isomorphism classes of $U(1)$ bundles on
the isomorphism classes of $\Mpc$-structures on a fixed symplectic vector
bundle. The map $P \mapsto c(P):=c_1(P_J(\lambda))$ at the level of
isomorphism classes is a bijection with $H^2(M,\Z)$. Up to isomorphism the
unique $\Mpc$-structure with $c(P)=0$ is  given by
\[
P_0(E,\omega,J):=U(E,\omega,J)\times_{U(V,\Omega,j)}\Mpc(V,\Omega,j)
\]
and any $\Mpc$-structure on $E$ is isomorphic to $L^{(1)}\cdot P_0(E,\omega,J)$
for  $L^{(1)}=P^{(1)}_J(\lambda)$.

\section{Spinors}

\begin{definition}
Given a principal $\Mpc(V,\Omega,j)$-bundle $P$ we form the associated
bundle $\mathcal{S} = \mathcal{S}(P,V,\Omega,j) =
P\times_{\Mpc(V,\Omega,j)} \H(V,\Omega,j)$ and similarly for bundles
$\mathcal{S}^\infty$ with fibre $ \H^\infty(V,\Omega,j)$ and
$\mathcal{S}^{-\infty}$ with fibre $ \H^{-\infty}(V,\Omega,j)$. Any of
these bundles we call a bundle of symplectic spinors associated to $P$.
\end{definition}

Remark that 
\begin{eqnarray*}
(L^{(1)}\cdot P)\times_{\Mpc(V,\Omega,j)} \H(V,\Omega,j) 
&\cong&
(L^{(1)}\times_M P_J) \times_{U(1)\times\MUc(V,\Omega,j)}\H(V,\Omega,j)\\
&\cong&
L \otimes (P_J\times_{\MUc(V,\Omega,j)} \H(V,\Omega,j))\\
&\cong&
L \otimes (P\times_{\Mpc(V,\Omega,j)} \H(V,\Omega,j))
\end{eqnarray*}
and thus 
\[
\mathcal{S}(L^{(1)}\cdot P,V,\Omega,j) = L \otimes \mathcal{S}(P,V,\Omega,j) 
\]
with similar statements for the bundles $\mathcal{S}^{\pm\infty}(P,V,\Omega,j)$.

Remark also that for  
$P_0(E,\omega,J)=U(E,\omega,J)\times_{U(V,\Omega,j)}\Mpc(V,\Omega,j)$
we have 
\[
\mathcal{S}_0:=\mathcal{S}(P_0(E,\omega,J)) 
= U(E,\omega,J)\times_{U(V,\Omega,j)} \H(V,\Omega,j),
\]
so that, in the general situation,  writing $P=L^{(1)}\cdot P_0$ the spinor
space is $L\otimes \mathcal{S}_0$.

The spinor bundle inherits a Hermitean structure from the one on the fibre:
\[
 h(\psi=[\xi,f]),\psi':=[\xi,  f']):=(f,f')_j.
\]

If $P$ is an $\Mpc$-structure on $(E,\omega)$ then
$E=P\times_{\Mpc(V,\Omega,j),\sigma}V$ acts on the space of spinors
$\mathcal{S}=P\times_{\Mpc(V,\Omega,j)} \H(V,\Omega,j)$ by Clifford
multiplication
\[
Cl \colon E\otimes \mathcal{S} \to \mathcal{S} : (e=[\xi, v])\otimes 
(\psi=[\xi,f])\mapsto Cl(e)\psi:=[\xi, cl(v) f];
\]
this is well defined because $cl(gv)Uf=Ucl(v)f$ for any $(U,g)$ in
$\Mpc(V,\Omega,j)$. When one has chosen a positive compatible almost
complex structure $J$ on $(E,\omega)$, we consider  
$E=P_J\times_{\MUc(V,\Omega,j),\sigma}V$,
$\mathcal{S}=P_J\times_{\MUc(V,\Omega,j)} \H(V,\Omega,j)$ and define in a
similar way annihilation and creation operators:
\[
A_J \colon E\otimes \mathcal{S} \to \mathcal{S} : 
(e=[\xi', v])\otimes (\psi=[\xi',f])\mapsto A_J(e)\psi:=[\xi', a(v) f];
\]
\[
C_J\colon E\otimes \mathcal{S} \to \mathcal{S} : (e=[\xi', v])\otimes 
(\psi=[\xi',f])\mapsto C_J(e)\psi:=[\xi', c(v) f];
\]
this is well defined because $a(gv)Uf=Ua(v)f$ for any $(U,g)$ in
$\MUc(V,\Omega,j)$. We have (cf prop\ref{prop:create})
$A_J(Je)=iA_J(e),~C_J(Je)=-iC_J(e),~h(A_J(e)\psi,
\psi')=h(\psi,C_J(e)\psi')$.

Since the centre of $\MUc$ acts trivially on the Heisenberg group and its
Lie algebra, the Clifford multiplication will commute with the process of
tensoring with a line bundle. If $s$ is a section of $L$ and $\psi$ is a
spinor for $P$ then $s\otimes \psi$ is a spinor for $L^{(1)}\cdot P$ and
\[ 
Cl(X)(s\otimes \psi) = s\otimes (Cl(X) \psi) 
\] 
with similar formulas for creation and annihilation operators.

By Proposition \ref{cor:muc}, when there is a positive
compatible almost complex structure $J$ on $(E,\omega)$, the grading of
$\H(V,\Omega,j)$ by polynomial degree in the complex variable in $(V,j)$
will pass canonically to the bundles associated to $P_J$ giving a dense
sub-bundle of the Hilbert and Fr\'echet bundles isomorphic to 
\[
L \otimes
\left(\oplus_p S^p((E')^*)\right)
\]
 where $L$ is the bundle associated to
$P_J$ by the character $\lambda$, and $E'$ denotes the $(1,0)$ vectors
of $J$ on the complexification $E \otimes_\R \C$. The term $L \otimes
S^0((E')^*) = L$ of degree zero in the grading is a copy of the line
bundle $L$. This can be identified as the common kernel of the
annihilation operators, namely the vacuum states in the Fock picture.

\section{$\Mpc$ Connections}\label{connections}

\begin{definition}
An $\Mpc$-connection in an $\Mpc$ structure $P$ on a symplectic vector
bundle $(E,\omega)$ is a principal connection $\alpha$ in $P$.
\end{definition}

At the level of Lie algebras, $\mpc(V,\Omega,j)$ splits as a sum $\u(1) +
\sp(V,\Omega)$ and so a connection 1-form $\alpha$ on $P$ can be split
$\alpha = \alpha_0 + \alpha_1$ where $\alpha_0$ is $\u(1)$-valued and
$\alpha_1$ is $\sp(V,\Omega)$-valued. $\alpha_0$ is then basic for
$\tilde\eta : P \rightarrow P^{(1)}(\eta)$ so is the pull-back of a
$\u(1)$-valued 1-form $\beta_0$ on $P^{(1)}(\eta)$  and $\alpha_1$ is the
pull-back under $\phi: P\rightarrow Sp(E,\omega)$ of a
$\sp(V,\Omega)$-valued 1-form $\beta_1$ on $Sp(E,\omega)$. Each form
$\beta_i$, $i=1,2$, is then a connection form on the corresponding bundle.
Thus an $\Mpc$-connection on $P$ induces connections in $E$ and $P(\eta)$.
The converse is true -- we pull back and add connection 1-forms in
$P^{(1)}(\eta)$ and $Sp(E,\omega)$ to get a connection 1-form on $P$.

Under twisting these connections behave compatibly. If we have a Hermitean
connection in $L^{(1)}$ and an $\Mpc$-connection in $P$ then there is an
induced $\Mpc$-connection in $L^{(1)}\cdot P$ and the connection induced in
$E$ is unchanged.

To calculate the curvature of a connection $\alpha$ in $P$ we observe that
$d\alpha = d\alpha_0 + d\alpha_1$ and $[\alpha\wedge\alpha] =
[\alpha_1\wedge\alpha_1]$ so that $d\alpha_0$ is a basic 2-form and
descends to an imaginary 2-form $i\omega^\alpha$ on $M$. We call
$\omega^\alpha$ the \textit{central curvature} of $\alpha$.

Another way to proceed is to take the covariant derivative induced in the 
line bundle $P(\eta)$ by the connection $\alpha$ and take its curvature
2-form, also an imaginary 2-form. To see how these 2-forms are related
we choose a local section $p \colon U \to P$ on some open set. Recalling
that $P(\eta) = P\times_{\Mpc(V,\Omega,j),\eta} \C$, we get a section
$s$ of $P(\eta)$ on $U$ by setting $s(x) = [p(x),1]$.
The covariant derivative $\nabla s$ is given
on $U$ by $
\nabla_X s = (\eta_*)(p^*\alpha(X))\,s$
so  that the curvature 2-form is then given on $U$ by\\
$
d ((\eta_*)(p^*\alpha)) = (\eta_*)(p^*d\alpha) = (\eta_*)(p^*d\alpha_0) 
$
since $\eta_*$ vanishes on brackets and thus on $\sp(V,\Omega)$. On the
centre of $\Mpc(V,\Omega,J)$,
$\eta$ is the squaring map, so $\eta_*(\alpha_0)= 2\alpha_0$. Thus
$(\eta_*)(p^*d\alpha_0) = 2p^* \pi^*( i\omega^\alpha) =  2i \omega^\alpha
$
since $\pi\circ p = \Id_U$. Since the curvature and the right hand side 
of this equation are globally defined we have shown:

\begin{theorem}
If $P$ is an $\Mpc$ structure on $(E,\omega)$ and $\alpha$ is a connection 
in $P$ with central curvature $\omega^\alpha$ then the connection induced
in the associated line bundle $P(\eta)$ has curvature $2i\omega^\alpha$.
\end{theorem}

\begin{remark}
The derivative of the map (\ref{eqn:twistgroups}) is addition so when we
take a connection 1-form $\gamma$ in a $U(1)$-bundle $L$ and an
$\Mpc$-connection $\alpha$, the connection form in the bundle $L\cdot P$
associated to the fibre product will be $\pi_1^*\gamma + \pi_2^*\alpha$, so
the net affect is to add $\gamma$ to the central component $\alpha_0$.
Hence the central curvature of the connection in $L\cdot P$ is the sum of
the curvature of $L$ with the central curvature of $P$.
\end{remark}

An $\Mpc$-connection $\alpha$ on $P$ induces a connection on $E$
(i.e. a covariant derivative on the space of its sections :
$\nabla: \Gamma(M,E)\rightarrow \Gamma(M,T^*M\otimes E)$) and a connection
on $\mathcal{S}$. Remark that the Clifford multiplication is parallel
 \[
 \nabla_X Cl(e)\psi=Cl(\nabla_X e)\psi + Cl(e)\nabla_X\psi.
 \]

A connection in the $\MUc$-structure will be an $\Mpc$ connection
inducing a connection on $E$ which preserves $\omega$ and $J$. Such a
connection induces one on the spinor bundle preserving the grading. The
maps $A_J$ and $C_J$ of annihilation and creation are parallel under
such a $\MUc$-connection  and lower and raise degrees by 1, hence
Clifford multiplication  mixes up degrees.

\section{Symplectic Dirac Operators}\label{Dirac}

In \cite{refs:Habermanns} a theory of symplectic Dirac operators is
developed based on metaplectic structures. However topologically being
metaplectic is the same as being spin and many interesting symplectic
manifolds such as $\C P^{2n}$ are not spin. We have seen that all
symplectic manifolds have $\Mpc$ structures, so all have spinors. We
shall define Dirac operators in the $\Mpc$ context analogous to
\cite{refs:Habermanns} and in addition make use of the extra structure
of the Fock space picture of the symplectic spinors.

In what follows $(M,\omega)$ will be a symplectic manifold and we apply the
theory of $\Mpc$ spinors to the symplectic vector bundle $(TM,\omega)$. Let
$J$ be a positive almost complex structure on $M$ compatible with $\omega$,
$SpFr(M,\omega)$ and $UFr(M,\omega,J)$ the symplectic and unitary frame
bundles where we have fixed some symplectic vector space $(V,\Omega)$ and
$j \in j_+(V,\Omega)$ with $\dim_{\R} V = \dim M$. Fix an $\Mpc$ structure
$P$ on $(TM,\omega)$ and let $P_J \subset P$ be the $\MUc$ reduction
determined by $J$. A connection $\alpha$ in $P_J$ determines a connection
in $P$ and so covariant derivatives in associated vector bundles such as
$TM$, $L=P_J(\lambda)$ and $\mathcal{S}$ which we denote by $\nabla$. For
these covariant derivatives $\omega$ and $Cl ,A_J,C_J$ are parallel, but
the covariant derivative in $TM$ may have torsion which we denote by
$T^\nabla$.

The (symplectic) Dirac operator is a first order differential operator
defined on sections of $\mathcal{S}$ as the contraction, using $\omega$, of
the Clifford multiplication and of the covariant derivative of spinor
fields. Taking a local frame field $e_i$ for $TM$, we form the dual frame
field $e^i$ which satisfies $\omega(e_i,e^j) =
\delta_i^j~~\textrm{i.e.}~e^j=-\sum_k\omega^{jk}e_k$ where $\omega^{ij}$
are the components of the matrix inverse to $\omega_{ij}:=\omega(e_i,e_j)$.
For $\psi \in \Gamma (\mathcal{S})$ we set
\[
D\psi := \sum_i Cl(e_i) \nabla_{e^i} \psi=-\sum_{ij} \omega^{ij}Cl(e_i) 
\nabla_{e_j} \psi
\]
which is easily seen to be independent of the choice of frame. In
\cite{refs:Habermanns} a second Dirac operator is defined using the
metric  to define the contraction of $Cl$ and $\nabla\psi$. Hence
\[
\wt{D}\psi :=  \sum_i Cl(Je_i) \nabla_{e^i} \psi=\sum_{ij} g^{ij}Cl(e_i) 
\nabla_{e_j} \psi.
\]
where $g^{ij}$ are 
the components of the matrix inverse to $g_{ij}:=\omega(e_i,Je_j)$.
In the presence of an almost complex structure $J$ it is convenient to
write derivatives in terms of their $(1,0)$ and $(0,1)$ parts. That is
we complexify $TM$ and then decompose $TM^{\C}$ into the $\pm i$
eigenbundles of $J$ which are denoted by $T'M$ and $T''M$. If $X$ is a
tangent vector then it decomposes into two pieces $X = X' + X''$ lying
in these two subbundles so $JX' = i X'$ and $JX'' = -i X''$. We can then
define
\[
{\nabla'}_X := \nabla_{X'}, \qquad {\nabla''}_X := \nabla_{X''}
\]
after extending $\nabla$ by complex linearity to act on complex vector
fields. We can now define two partial Dirac operators $D'$ and $D''$ by
using these operators instead of $\nabla$
\[
D'\psi =  \sum_i Cl(e_i) {\nabla'}_{e^i} \psi,\qquad
D''\psi =  \sum_i Cl(e_i) {\nabla''}_{e^i} \psi.
\]
\begin{proposition}
\[
D = D' + D'', \qquad \wt{D} = -i D' + i D''.
\]
\end{proposition}
\begin{proof}
The first is obvious since $\nabla = \nabla' + \nabla''$. For the second
we observe that if $e_i$ is a frame, so is $Je_i$ and since $J$
preserves $\omega$, the dual frame of $Je_i$ is $Je^i$. Thus $\wt{D}
\psi =  \sum_i Cl(Je_i) \nabla_{e^i} \psi = -\sum_i Cl(e_i)
\nabla_{Je^i} \psi = - \sum_i Cl(e_i)( {\nabla'}_{Je^i} \psi +
{\nabla''}_{Je^i}\psi)$. But $(JX)' = i X'$ and $(JX)''=  -i X''$ giving
the result.
\end{proof}
A nice thing happens due to the behaviour of $A_J$ and $C_J$  with respect
to $J$:
\begin{proposition}
\[
D'\psi =  \sum_i C_J(e_i) {\nabla}_{e^i} \psi= -\sum_{kl}\omega^{kl} C_J(e_k) 
{\nabla}_{e_l} \psi  =  \sum_i C_J(e_i) {\nabla'}_{e^i} \psi,
\]
\[
D''\psi =  - \sum_i A_J(e_i) {\nabla}_{e^i} \psi=\sum_{rs}\omega^{rs} A_J(e_r) 
{\nabla}_{e_s} \psi   =  - \sum_i A_J(e_i) {\nabla''}_{e^i}\psi .
\]
\end{proposition}

\begin{proof}
We show the result for $D'$; the second result is similar.
\begin{eqnarray*}
D'\psi &=& \half  \sum_j Cl(e_j) {\nabla}_{(e^j-iJe^j)}\psi
= \frac12 \bigl(\sum_j Cl(e_j){\nabla}_{e^j}\psi +
i \sum_j Cl(Je_j){\nabla}_{e^j}\psi\bigr)\\
&=& \frac12  \sum_j \bigl( Cl(e_j)  +iCl(Je_j)\bigr){\nabla}_{e^j}\psi
= \sum_iC_J(e_i){\nabla}_{e^i}\psi.
\end{eqnarray*}
On the other hand $
\sum_jC_J(e_j){\nabla}_{e^j}\psi =\sum_jC_J(Je_j){\nabla}_{Je^j}\psi
= -i\sum_jC_J(e_j){\nabla}_{Je^j}\psi$ 
so
\[
D'\psi = \frac12 \left(\sum_jC_J(e_j){\nabla}_{e^j}\psi
-i \sum_jC_J(e_j){\nabla}_{Je^j}\psi\right)
= \sum_j C_J(e_j) {\nabla'}_{e^j} \psi.
\]
\end{proof}

The second order operator $\mathcal{P}$ defined in \cite{refs:Habermanns}
by  $\mathcal{P} = i[\wt{D},D]$ is now given by $\mathcal{P} = i[-i D'
+iD'',D' + D''] = 2[D',D'']$. Observing that $D'$ raises the Fock degree
by 1 whilst $D''$ lowers it by 1, it is clear that $\mathcal{P}$ preserves
the Fock degree. Thus, 
\begin{proposition}
On the dense subspace of polynomial spinor fields, the operator
$\mathcal{P}:=2[D',D'']$ is a direct sum of operators acting on sections of
finite rank vector bundles.
\end{proposition}
One defines in a natural way a $L^2$-structure on the space of sections of
the spinor bundle $\mathcal{S}$ and on the space of sections of
$T^*M\otimes\mathcal{S}$:
\[ 
<\psi,\psi'>:=\int_M h(\psi,\psi')\frac{\omega^n}{n!},\quad
<\gamma,\gamma'>:=\int_M \sum_{ab}g^{ab}h(\gamma(e_a),\gamma'(e_b))\frac{\omega^n}{n!}
\]
for smooth sections with compact support, where $\{e_a\}$ is a  local frame
field and where, as before, $g^{ab}$ are the components of the matrix
which is the inverse of the matrix $(g_{ab})$ with
$g_{ab}:=g(e_a,e_b):=\omega(e_a,Je_b)$.

If $X$ is a vector field on $M$  its divergence is the trace of  its 
covariant derivative:
$
\div^\nabla X:=\Tr [Y\mapsto \nabla_YX].
$
One defines the torsion-vector field $\tau^\nabla$:
\[
\tau^\nabla=\half \sum_k T^\nabla(e_k,e^k)\text{\ \ with, as before,\ \ }
\omega(e_k,e^l)=\delta_k^l.
\]
One has $\omega(Z,\tau^\nabla)=\Tr [Y\mapsto T^\nabla(Y,Z)]$.  Indeed the
sum over cyclic permutations of $X,Y,Z$ of $\omega(T^\nabla(Y,Z),X)$
vanishes and
$
\Tr [Y\mapsto T^\nabla(Y,Z)]=\sum_k\omega(T^\nabla(e_k,Z),e^k)$\\
$=\half\sum_k\left(\omega(T^\nabla(e_k,Z),e^k) + \omega(T^\nabla(Z,Xe^k),e_k)\right)$
$=\half \omega(T^\nabla(e_k,e^k),Z)$.
\begin{lemma}
$\mathcal{L}_X\omega^n=(\div^\nabla X+\omega(X,\tau^\nabla))\omega^n$.
\end{lemma}
\begin{proof}
For any $2$-form $\alpha$, one has
$\alpha\wedge\omega^{n-1}=\left(\frac{1}{2n}\sum_k\alpha(e_k,e^k)\right)\,
\omega^n$. On the other hand
$\mathcal{L}_X\omega^n=n\left(\mathcal{L}_X\omega\right)\wedge\omega^{n-1}$
and
\begin{eqnarray*}
\sum_k\mathcal{L}_X\omega(e_k,e^k)&=&-\sum_k\omega ([X,e_k],e^k)-\sum_k\omega (e_k,[X,e^k])\\
&=&\sum_k\omega (\nabla_{e_k}X,e^k)+\sum_k\omega (e_k,\nabla_{e^k}X)\\
& &\qquad+\sum_k\omega (T^\nabla(X,e_k),e^k)+\sum_k\omega (e_k,T^\nabla(X,e^k))\\
&=&2\div^\nabla X -\sum_k\omega (T^\nabla(e_k,e^k),X)=2\div^\nabla X + 2\omega(X,\tau^\nabla).
\end{eqnarray*}
\end{proof}
\begin{proposition}
Given a $\MUc$-structure and a $\MUc$-connection on  $(M,\omega)$, taking
any local frame field $\{e_a\}$ of the tangent bundle, we have, for
compactly supported smooth sections:
 \[
 <D'\psi,\psi'>=  <\psi,\left(D''+A_J(\tau^\nabla)\right)\psi'>
 \]
 \[
< \nabla \psi, \beta > = <\psi, \nabla^* \beta > \textrm{ with }
\nabla^* \beta := -\sum_{ab}g^{ab} (\nabla_{e_a}\beta) (e_b) + \beta (J\tau^\nabla)
\]
The Laplacian on spinors is thus given by 
$
\nabla^*\nabla \psi = - g^{ab} \nabla^2 \psi (e_a,e_b) + \nabla_{J\tau^\nabla} \psi
$
where $\nabla^2 \psi (e_a,e_b) := \nabla_{e_a}(\nabla_{e_b}\psi)-\nabla_{\nabla_{e_a}e_b} \psi$.
\end{proposition}
\begin{proof}
Indeed 
\begin{eqnarray*}
 <D'\psi,\psi'>
 &=&<\sum_{k}C_J(e_k) {\nabla}_{e^k} \psi ,\psi'>
 =\sum_{k}\int_M h(C_J(e_k) {\nabla}_{e^k} \psi ,\psi')\frac{\omega^n}{n!}\\
 &=&\sum_{k}\int_M h({\nabla}_{e^k} \psi ,A_J(e_k)\psi')\frac{\omega^n}{n!}\\
 &=&-\sum_{k}\int_M h( \psi ,{\nabla}_{e^k}(A_J(e_k)\psi'))\frac{\omega^n}{n!}
 -\sum_{k}\int_M h( \psi ,A_J(e_k)\psi')\frac{\mathcal{L}_{e^k}{\omega^n}}{n!}\\
 &=&<\psi,D''\psi'>\\
 &&-\sum_{k}\int_M h( \psi ,\left((A_J({\nabla}_{e^k}e_k)+(\div^\nabla e^k
 +\omega(e^k,\tau^\nabla))A_J(e_k)\right)\psi'))\frac{\omega^n}{n!}.
 \end{eqnarray*}
 Since $\sum_k {\nabla}_{e^k}e_k=-\sum_r(\div^\nabla  e^r) e_r$, we get the
 result.  To get the formula for $\nabla^*$:
\begin{eqnarray*}
< \nabla \psi, \beta >&=&\int_M \sum_{ab}g^{ab}h(\nabla_{e_a}\psi,\beta(e_b))
\frac{\omega^n}{n!}\\&=&\int_M \sum_{ab}g^{ab}\left({e_a}\left(h(\psi,\beta(e_b))\right)
-h(\psi,\nabla_{e_a}(\beta(e_b)))\right)\frac{\omega^n}{n!}\\
&=&-\int_M \sum_{ab}h(\psi,\beta(e_b))\bigl(e_a\left(g^{ab}\right)\frac{\omega^n}{n!}
+g^{ab}\frac{\mathcal{L}_{e_a}\omega^n}{n!}\bigr)\\
&&\quad -<\psi,\sum_{ab}g^{ab}\nabla_{e_a}(\beta(e_b))>.\\
\end{eqnarray*}
But $\mathcal{L}_{e_a}\omega^n=\left(\div^\nabla e_a 
+\omega(e_a,\tau^\nabla)\right)\omega^n$ 
and $\div^\nabla e_r=\sum_a(\nabla_{e_a}e_r)^a$; since $\nabla g=0$ we have
 $e_a\left(g^{ab}\right)=-\sum_r\left( g^{ar}(\nabla_{e_a}e_r)^b
 + g^{rb}(\nabla_{e_a}e_r)^a\right)$; 
 hence
\begin{eqnarray*}
 < \nabla \psi, \beta >&=&\int_M \sum_{ar}g^{ar} h(\psi,\beta(\sum_b 
 (\nabla_{e_a}e_r)^be_b))\frac{\omega^n}{n!}\\
&&-\int_M \sum_{ab}h(\psi,\beta(e_b))g^{ab}\omega(e_a,\tau^\nabla)
\frac{\omega^n}{n!} -<\psi,\sum_{ab}g^{ab}\nabla_{e_a}(\beta(e_b))>\\
&=& <\psi,\left( -\sum_{ab}g^{ab} (\nabla_{e_a}\beta) (e_b) 
+ \beta (J\tau^\nabla)\right)\beta>.
\end{eqnarray*}
The formula for the Laplacian follows readily.
\end{proof}
 \begin{proposition} The operator $\mathcal{P}=2[D',D'']$ is elliptic and one has:
\begin{eqnarray*}
[D',D'']&=&-\frac{1}{2\hbar} \nabla^*\nabla \psi 
+\frac{1}{2\hbar} \nabla_{J\tau^\nabla} \psi\\
    && +\half\sum\omega^{kl}\omega^{rs} (C_J(e_k)A_J(e_r)-A_J(e_k)C_J(e_r))
    \bigl(R^\nabla(e_l,e_s)\psi - \nabla_{T^\nabla(e_l,e_s)}\psi\bigr)\\
\end{eqnarray*}
where $R^\nabla$ denotes the curvature, i.e. $ R^\nabla(X,Y)\psi=\nabla_X(\nabla_Y\psi)-\nabla_Y(\nabla_X\psi)-\nabla_{[X,Y]}\psi$. The last term can be written 
$+\frac{ i }{2} \sum\omega^{kl}\omega^{rs}Cl( e_k )Cl(Je_r) \bigl(R^\nabla(e_l,e_s)\psi 
- \nabla_{T^\nabla(e_l,e_s)}\psi\bigr)$.
 \end{proposition}
 \begin{proof}
 Since $\nabla A_J=0$ and $\nabla C_J=0$ we have
\begin{eqnarray*}
 [D',D'']&=&[- \sum_{kl}\omega^{kl}C_J(e_k) {\nabla}_{e_l},   
 \sum_{rs} \omega^{rs}A_J(e_r) {\nabla}_{e_s}]\\
 &=&- \sum_{klrs}\omega^{kl}\omega^{rs}\left(C_J(e_k)A_J(e_r)
 \nabla^2_{e_l e_s}-A_J(e_r)C_J(e_k)\nabla^2_{e_s e_l}\right)\\
  &=&- \sum_{klrs}\omega^{kl}\omega^{rs}\left(C_J(e_k)A_J(e_r)
  -A_J(e_k)C_J(e_r)\right)\nabla^2_{e_s e_l}\\
  &=&-\half \sum_{klrs}\omega^{kl}\omega^{rs}\left(C_J(e_k)
  A_J(e_r)-A_J(e_k)C_J(e_r)\right)\left(\nabla^2_{e_s e_l}+\nabla^2_{e_l e_s}\right)\\
  &&- \half\sum_{klrs}\omega^{kl}\omega^{rs}\left(C_J(e_k)
  A_J(e_r)-A_J(e_k)C_J(e_r)\right)\left(\nabla^2_{e_s e_l}-\nabla^2_{e_l e_s}\right).\\
  \end{eqnarray*}
The first term is also
$-\half\sum_{klrs}\omega^{kl}\omega^{rs}\left(C_J(e_k)A_J(e_r)
-A_J(e_r)C_J(e_k)\right)(\nabla^2_{e_s e_l}+\nabla^2_{e_l e_s})$;  \\ 
since 
 $-[C_J(e_k),A_J(e_r)] =\frac{1}{2\hbar}\langle e_r,e_k\rangle_j 
 =\frac{1}{2\hbar}(g_{rk}-i\omega_{rk})$ and 
 $\sum_{kr}\omega^{kl}\omega^{rs}g_{rk}=g^{ls}$, it is equal to 
 $\frac{1}{2\hbar}\sum_{ls}g^{ls}\nabla^2_{e_s e_l}
 =\frac{1}{2\hbar}(-\nabla^*\nabla+\nabla_{J\tau^\nabla})$.\\
For the second term, we observe that  $\nabla^2_{e_s e_l}-\nabla^2_{e_l
e_s}=R^\nabla(e_s, e_l)-\nabla_{T^\nabla(e_s,e_l)}$.
 \end{proof}
On any symplectic manifold with a chosen positive $\omega$-compatible
almost complex structure $J$, there are linear connections such that
$\nabla\omega=0, \nabla J=0$ and $\tau^\nabla=0$. Indeed, if $\nabla^1$ is a
linear connection such that $\nabla^1\omega=0, \nabla^1J=0$ we set
\[
\nabla_XY:=\nabla^1_XY-\frac{1}{2n}\left(\omega(\tau^{\nabla^1}\!\!,Y)X
+\omega(X,Y)\tau^{\nabla^1}\!\!+\omega(J\tau^{\nabla^1}\!\!,Y)JX
+\omega(JX,Y)J\tau^{\nabla^1}\right).
\]

\newpage 


\begin{thebibliography}{99}

\bibitem{refs:Bargmann}
V. Bargmann,
On a Hilbert space of analytic functions and an associated integral transform. 
\textit{Comm. Pure Appl. Math.} 14 (1961) 187--214;
Remarks on a Hilbert space of analytic functions, 
\textit{Proc. Nat. Acad. Sci. U.S.A.} 48 1962 199--204;
On a Hilbert space of analytic functions and an associated integral transform, 
Part II. A family of related function spaces. Application to distribution theory, 
\textit{Comm. Pure Appl. Math.} 20 (1967) 1--101.

\bibitem{refs:Berezin}
F.A. Berezin, 
\textit{The method of second quantization}, 
Pure and Applied Physics, Vol.~24 (Academic Press, New York-London, 1966).

\bibitem{refs:Fock}
V.A. Fock,
Konfigurationsraum und zweite Quantelung,
\textit{Z. Phys.} 75 (1932), 622--647.

\bibitem{refs:Folland}
G.B. Folland, 
\textit{Harmonic Analysis in Phase Space}, 
Annals of Math. Study~122, Appendix~A, Thm.~3, p.~258 
(Princeton University Press, Princeton, NJ, 1989).

\bibitem{refs:ForgerHess}
M. Forger and H. Hess,
Universal Metaplectic Structures and Geometric Quantization,
\textit{Commun. Math. Phys.}, 64, (1979) 269--278.

\bibitem{refs:Habermanns}
K. Habermann and L. Habermann,
\textit{Introduction to Symplectic Dirac Operators},
Lecture Notes in Mathematics 1887, 
(Springer-Verlag, Berlin Heidelberg New York, 2006).

\bibitem{refs:Itzykson}
C. Itzykson, 
Remarks on boson commutation rules. 
\textit{Commun. Math. Phys.}, 4, (1967) 92--122.

\bibitem{refs:Kostant}
B. Kostant, 
\textit{Symplectic Spinors.} 
Symposia Mathematica, vol.~XIV, pp.~139--152 
(Cambridge University Press, Cambridge, 1974).

\bibitem{refs:RobRaw}
P.L. Robinson and J.H. Rawnsley, 
\textit{The metaplectic representation, $\Mpc$ structures and 
geometric quantization.} 
Memoirs of the A.M.S. vol.~81, no.~410. 
(AMS, Providence RI, 1989).

\bibitem{refs:Segal}
I.E. Segal,
Lectures at the
\textit{1960 Boulder Summer Seminar}, 
(AMS, Providence, RI, 1962)

\bibitem{refs:Shale}
D. Shale,
Linear symmetries of free boson fields, 
\textit{Trans. Amer. Math. Soc.} 103 (1962) 149--167.

\bibitem{refs:Weil}
A. Weil,
Sur certains groupes d'op\'erateurs unitaires,
\textit{Acta Math.} 111 (1964) 143--211.

\end{thebibliography}
\end{document}